\documentclass[11pt,a4paper]{article}

\usepackage[english]{babel}
\usepackage[T1]{fontenc}

\usepackage[top=3cm,bottom=2.5cm,left=3cm,right=3cm,marginparwidth=1.75cm]{geometry}

\usepackage{amsmath}
\usepackage{amsthm}
\usepackage{amssymb}
\usepackage{cancel,shuffle}
\usepackage{stackrel}
\usepackage{enumerate}
\usepackage{graphicx}
\usepackage{fdsymbol}
\usepackage[all]{xy}
\usepackage[colorinlistoftodos]{todonotes}
\usepackage[colorlinks=true, allcolors=blue]{hyperref}
\usepackage{tikz-cd}
\tikzcdset{
arrow style=Latin Modern,
column sep/normal=3.6 em,
row sep/normal=2.7 em
}
\usepackage{float}

\newtheorem{theorem}{Theorem}[section]
\newtheorem{proposition}[theorem]{Proposition}
\newtheorem{lemma}[theorem]{Lemma}
\newtheorem{corollary}[theorem]{Corollary}
\newtheorem{remark}[theorem]{Remark}
\newtheorem{example}[theorem]{Example}

\newtheorem{definition}[theorem]{Definition}

\newfloat{diag}{ht}{diag}[section]
\floatname{diag}{Diagram}

\makeatletter
\newtheorem*{rep@theorem}{\rep@title}
\newcommand{\newreptheorem}[2]{%
\newenvironment{rep#1}[1]{%
 \def\rep@title{#2 \ref{##1}}%
 \begin{rep@theorem}}%
 {\end{rep@theorem}}}
\makeatother
\newreptheorem{theorem}{Theorem}

\def\li{\mathtt{i}}
\def\lj{\mathtt{j}}
\def\lk{\mathtt{k}}
\def\w{\mathtt{w}}
\def\la{\mathtt{a}}
\def\lb{\mathtt{b}}
\def\v{\mathtt{v}}
\def\u{\mathtt{u}}
\def\e{\mathtt{e}}
\def\itensor{\smallblackcircle}
\def\linearmap{B}
\newcommand{\Tminus}[1]{\mathop{\mathrm{T}^{-}_{#1}}}
\newcommand{\Tplus}[1]{\mathop{\mathrm{T}^{+}_{#1}}}

\def\halfshuffle{\mathbin{\succ}}
\def\varhalfshuffle{\mathbin{\curlyeqsucc}}
\def\hsm{\Lambda}

\begin{document}

\title{Signatures of paths transformed by polynomial maps}
\author{Laura Colmenarejo\footnote{Laura Colmenarejo, UMass Amherst (USA), \texttt{laura.colmenarejo.hernando@gmail.com}}\ \ \  and\ \  Rosa Prei\ss\footnote{Rosa Prei\ss , Technical University Berlin (Berlin, Germany), \texttt{preiss@math.tu-berlin.de}}}
\date{}
\maketitle{}

\begin{abstract}
We characterize the signature of piecewise continuously differentiable paths transformed by a polynomial map in terms of the signature of the original path. For this aim, we define recursively an algebra homomorphism between two shuffle algebras on words. This homomorphism does not depend on the path and behaves well with respect to composition and homogeneous maps. It allows us to describe the relation between the signature of a piecewise continuously differentiable path and the signature of the path obtained by transforming it under a polynomial map. We also study this map as a half-shuffle homomorphism and give a generalization of our main theorem in terms of Zinbiel algebras. 
\end{abstract}

\emph{Keywords}: Signature tensors, iterated integrals, tensor algebra, shuffle product, polynomial maps.

\section{Introduction}

In the 1950s, K. T. Chen introduced the \emph{iterated-integral signature} of a piecewise continuously differentiable path, which up to a natural equivalence relation, determines the initial path. In general, the signature of a path can be seen as a multidimensional time series. When the terminal time is fixed, the signature of a path can be seen as tensors and the calculation of the signature becomes a standard problem in data science. In \cite{PSS18}, M. Pfeffer, A. Seigal, and B. Sturmfels study the inverse problem: given partial information from a signature, can we recover the path? They consider signature tensors of order three under linear transformations and establish identifiability results and recovery algorithms for piecewise linear paths, polynomial paths, and generic dictionaries.

Coming from stochastic analysis, the signatures are becoming more relevant in other areas, such as algebraic geometry and combinatorics, and we would like to highlight some recent work. 
For instance, in \cite{DR18}, J. Diehl and J. Reizenstein offer a combinatorial approach to the understanding of invariants of multidimensional time series based on their signature. Another reference is \cite{AFS18}, in which C. Am\'endola, P. Friz, and B. Sturmfels look at the varieties of signatures of tensors for both deterministic and random paths, focusing on piecewise linear paths and polynomials paths, among others. Answering one of their questions, in \cite{G18}, F. Galuppi looks at rough paths, for which their signature variety shows surprising analogies with the Veronese variety.

In stochastic analysis, the study of the signatures  of paths arises in the theory of rough paths, where  \cite{FV10,FH14} are textbook references. Iterated integrals and the non-commutative series that encode them have also arisen in a variety of contexts in geometry and arithmetic, including the work of R. Hain in \cite{H02}, M. Kapranov in \cite{K09}, and J. Balakrishnan in \cite{B13}. The results we derive in this paper have the potential for future applications in all of these contexts.

Let us now present our problem and our main two results, Theorems \ref{thm: mainquestion} and \ref{thm: main}.

A piecewise continuously differentiable path $X$ in $\mathbb{R}^d$ is a map defined by $d$ piecewise smooth functions $X^{i}(t)$ in a parameter $t\in[0,L]$, for $i=1,\dots,d$. Its signature stores the collection of all the iterated integrals of the path $X$, which are of the form 
\begin{eqnarray}\label{iteratedintegral}
\int_0^L \int_0^{r_n}\dots\int_0^{r_2} dX_{r_1}^{i_1}\dots dX_{r_n}^{i_n},
\end{eqnarray}
where $X^i_r := X^i(r)$. The iterated integral \eqref{iteratedintegral} is a real number and it is associated to the sequence $(i_1,i_2\dots,i_n)$, for which the order is relevant. Therefore, we consider the signature $\sigma(X)$ as an element of $\mathrm{T}((\mathbb{R}^d))$, the space of formal power series in words in the alphabet $\{\mathtt{1},\mathtt{2},\dots,\mathtt{d}\}$. This space becomes an algebra with the concatenation product, denoted by the symbol $\itensor$. Its algebraic dual, denoted by $\mathrm{T}(\mathbb{R}^d)$, is the space of non-commutative polynomials in the same set of words. It is a commutative algebra with the shuffle product, which is denoted by $\shuffle$ and interleaves two words in all order-preserving ways, \cite{Reu93}. 

We also consider the following duality paring in $\mathrm{T}((\mathbb{R}^d))\times \mathrm{T}(\mathbb{R}^d)$:
\begin{eqnarray}\label{eq:scalar product}
\displaystyle{\left\langle \sum_{\w \in \mathcal{W}_d} a_{\w}\w, \v \right\rangle = a_{\v},}
\end{eqnarray}
where $\mathcal{W}_d$ denotes the set of words in the alphabet $\{\mathtt{1},\dots, \mathtt{d}\}$, together with the empty word $\e$.

Let $X$ be a piecewise continuously differentiable path in $\mathbb{R}^d$ and $\sigma(X)$ be its signature. Consider a polynomial map $p$ from $\mathbb{R}^d$ to $\mathbb{R}^m$. One can compute the image path $p(X)$ and ask for its signature, $\sigma(p(X))$. Then, the following question comes up: 
\begin{center}
\emph{How are both signatures, $\sigma(X)$ and $\sigma(p(X))$, related?}
\end{center}
We approach this question from an algebraic point of view. We consider the dual map $p^*: \mathbb{R}[x_1,\dots,x_m] \longrightarrow \mathbb{R}[x_1,\dots,x_d]$, where both sets of variables are commutative. It is natural and common to embed the polynomial ring $\mathbb{R}[x_1,\dots,x_m]$ into the tensor algebra $\left(\mathrm{T}(\mathbb{R}^m),\shuffle\right)$. For that we identify the variable $x_i$ with the letter $\li$ and we define the embedding, denoted by $\varphi_m$ (or $\varphi$), by sending the monomial $x_{i_1}\cdots x_{i_l}$ to the shuffle product $\li_1 \shuffle\cdots\shuffle \li_l$, for $1\leq i_1,\dots,i_l\leq m$, and extending by linearity. By construction, this map is a morphism of commutative algebras, and it is injective but not surjective. For instance, for $t\geq 2$, $\varphi(x_1\cdot x_2) = \mathtt{1}\shuffle \mathtt{2} = \mathtt{12} + \mathtt{21}$ and there is no other way to obtain the words $\mathtt{12}$ and $\mathtt{21}$ as $\varphi_d(h)$, for any polynomial $h\in \mathbb{R}[x_1,\dots, x_d]$. Therefore, we cannot find a polynomial in $ \mathbb{R}[x_1,\dots,x_d]$ with image $\mathtt{12}$. 

Our first step is to define a map $M_p:\left(\mathrm{T}(\mathbb{R}^m),\shuffle\right) \longrightarrow \left(\mathrm{T}(\mathbb{R}^d),\shuffle\right)$, which is an algebra homomorphism and that is unique in the following sense.
\begin{theorem}\label{thm: mainquestion}
There exists an algebra homomorphism $M_p: \left(\mathrm{T}(\mathbb{R}^m),\shuffle\right) \longrightarrow \left(\mathrm{T}(\mathbb{R}^d),\shuffle\right)$ such that its restriction $\left. M_p\right|_{Im(\varphi_m)}$ is the unique algebra homomorphism that makes the following diagram commute:
\begin{eqnarray}\label{diagram}
\xymatrix @R=12mm @C=8mm {
 \mathbb{R}[x_1,\dots, x_m]\ar[rr]^{p^*}\ar@{^{(}->}[d]_{\varphi_m} & & \mathbb{R}[x_1,\dots, x_d]\ar@{_{(}->}[d]_{\varphi_d} \\
Im(\varphi_m)\ar@{.>}[rr]^{\left. M_p\right|_{Im(\varphi_m)}}  & & Im(\varphi_d) \\
\save[]+<0mm,+12mm>*{\cap}\restore & & \save[]+<0mm,+12mm>*{\cap}\restore  \\
\save[]+<0mm,+21mm>*{\left(\mathrm{T}(\mathbb{R}^m),\shuffle\right)}   \restore & & \save[]+<0mm,+21mm>*{\left(\mathrm{T}({\mathbb{R}}^d),\shuffle\right)}\restore \\
}
\end{eqnarray}
\end{theorem}
\vspace{-2cm}The map $M_p$ has some further interesting properties and is the key to relate the signature of a path with the signature of its transformation under a polynomial map.
\begin{theorem}\label{thm: main}
Let $X:[0,L] \longrightarrow \mathbb{R}^d$ be a piecewise continuously differentiable path with $X(0)=0$ and let $p:\mathbb{R}^d \longrightarrow \mathbb{R}^m$ be a polynomial map with $p(0)=0$. Then, for all $\w \in \mathrm{T}(\mathbb{R}^m)$, 
\begin{eqnarray*}
\left\langle \sigma (p(X) ), \w \right\rangle = \left\langle \sigma(X), M_p(\w)\right\rangle.
\end{eqnarray*}
Equivalently, $\sigma(p(X)) = M_p^*\left(\sigma(X)\right)$. 
\end{theorem}
The paper is organized as follows. In Section \ref{sec:2}, we introduce briefly the framework of signatures of paths, as well as the basic notions that we need for our key combinatorial objects, the words. In Section \ref{sec:3}, we define the map $M_p$ and prove its properties in Proposition \ref{prop:propertiesMp}. Moreover, we present our main theorems, Theorems \ref{thm: mainquestion} and \ref{thm: main}, together with a generalization of the last one, Corollary \ref{cor: different origin}. In Section \ref{Sec:Halfshuffle} we look at $M_p$ as a half-shuffle homomorphism and give a generalization of Theorem \ref{thm: main} in terms of Zinbiel algebras. In Section \ref{Sec:4}, we also present two examples and a few consequences, Corollaries \ref{cor:polynomialpath}--\ref{cor:dualmap}. Finally, Section \ref{sec:5} is dedicated to applications and future work.

\section{Signatures of paths and words}\label{sec:2}

Given a piecewise continuously differentiable path $X:[0,L]\longrightarrow \mathbb{R}^d$, for any $i_1,\dots, i_n\in \{1,2,\dots, d\}$ the following integral is classically well-defined
\begin{eqnarray*}
\int_0^L dX^{i_1}\dots dX^{i_n} := 
\int_0^L \int_0^{r_n}\dots\int_0^{r_2} dX_{r_1}^{i_1}\dots dX_{r_n}^{i_n}=
 \int_0^L\int_0^{r_n}\dots\int_0^{r_2} \dot{X}_{r_1}^{i_1}\dots \dot{X}_{r_n}^{i_n} dr_1\dots dr_n.
\end{eqnarray*}
We would like to store the collection of all these integrals. 
\begin{definition}
The \emph{signature of} $X$ is defined as the following formal power series
\begin{eqnarray*}
\sigma(X) = \sum_{n\geq 0} \sum_{i_1\dots i_n} \underbrace{\int_0^L \int_0^{r_n}\dots\int_0^{r_2} dX_{r_1}^{i_1}\dots dX_{r_n}^{i_n} }_{\in \mathbb{R}}\cdot \li_1 \cdots \li_n \in \mathrm{T}((\mathbb{R}^d)).
\end{eqnarray*}
\end{definition}
As we mention in the introduction, $\mathrm{T}((\mathbb{R}^d))$ is the space of formal power series in words in the alphabet $\{\mathtt{1},\dots,\mathtt{d}\}$, and we denote by $\e$ the empty word.  It is an algebra with the concatenation product, denoted by $\w\itensor \v$ (or simply $\w\v$), which is well-defined since it respects the grading given by the number of letters appearing in each word. 
We also consider its algebraic dual $\mathrm{T}(\mathbb{R}^d)$, which is the set of polynomials in words in the same alphabet. The algebra $\mathrm{T}(\mathbb{R}^d)$ has the concatenation product, which is the same as for $\mathrm{T}((\mathbb{R}^d))$ if we multiply two finite power series.  However, we consider $\mathrm{T}(\mathbb{R}^d)$ as an algebra with the shuffle product, which we define recursively as follows. 
\begin{definition}\label{recdef-shuffle}
Let $\w$, $\w_1$ and $\w_2$ be three words and $\la$ and $\lb$ two letters. We define the \emph{shuffle product} of two words recursively by
\begin{eqnarray*}
\begin{array}{l}
\e \shuffle \w = \w \shuffle \e = \w, \text{ and }\\
(\w_1\itensor \la) \shuffle (\w_2\itensor \lb) = \left(\w_1\shuffle (\w_2\itensor \lb)\right)\itensor \la + \left((\w_1\itensor \la )\shuffle \w_2\right)\itensor\lb.
\end{array}
\end{eqnarray*}
\end{definition}
Note that in the shuffle product, we distinguish duplicated letters. For instance, for a letter $\la$, we have $\la \shuffle \la = 2 \cdot \la\la$. Notice that the concatenation is a non-commutative operation, whereas the shuffle product is commutative. 

We also need a few notions on words.  The \emph{length of a word} $\w$ is denoted by $\ell(\w)$ and counts the number of letters in $\w$. We extend this definition by linearity, defining $\ell(\w_1+\w_2) := \max\{\ell(\w_1),\ell(\w_2)\}$, for any words $\w_1$ and $\w_2$. Therefore, $\ell(\w_1\itensor \w_2) = \ell(\w_1) + \ell(\w_2) = \ell(\w_1 \shuffle \w_2)$. As an $\mathbb{R}$-vector space, $\mathrm{T}(\mathbb{R}^d)$ is graded by the length of the words:
\begin{eqnarray*}
\mathrm{T}(\mathbb{R}^d) = \bigoplus_{n\geq 0} \mathrm{T}^n(\mathbb{R}^d),
\end{eqnarray*}
where $\mathrm{T}^n(\mathbb{R}^d)$ is the vector space spanned by the words of length $n$. We also denote by $\mathrm{T}^{\leq n}(\mathbb{R}^d)$ the partial direct sum $\displaystyle{\bigoplus_{k\leq n} \mathrm{T}^k(\mathbb{R}^d)}$. This notation extends to $\mathrm{T}((\mathbb{R}^d))$. The same way, $\sigma^{(n)}(X)$ denotes the partial sum of $\sigma(X)$ for which all the appearing words have length exactly $n$. We are ready to prove the following result.

\begin{proposition}\label{prop:shuffle_associativity}
 The shuffle product is associative.
\end{proposition}
\begin{proof}
The associativity is clear for the empty word since $(\e\shuffle\e)\shuffle\e=\e=\e\shuffle(\e\shuffle\e)$. Now, we proceed by induction. Assume that for any words $\w_1$, $\v_1$, and $\u_1$ such that $\ell(\w_1)+\ell(\v_1)+\ell(\u_1)=n$, for some $n\in\mathbb{N}_0$, we have that $(\w_1\shuffle\v_1)\shuffle\u_1=\w_1\shuffle(\v_1\shuffle\u_1)$. This is our inductive hypothesis. 
 
Let $\w_2,\v_2,\u_2$ be arbitrary words with the property that $\ell(\w_2)+\ell(\v_2)+\ell(\u_2)=n+1$. At least one of those words must thus be non-empty. If exactly two of the words are empty, both $(\w_2\shuffle\v_2)\shuffle\u_2$ and $\w_2\shuffle(\v_2\shuffle\u_2)$ are obviously equal to the non-empty word. If exactly one of the words is empty, both $(\w_2\shuffle\v_2)\shuffle\u_2$ and $\w_2\shuffle(\v_2\shuffle\u_2)$ are obviously equal to the shuffle product of the two non-empty words. In the remaining case, if $\w_2,\v_2,\u_2$ are all non-empty, there are words $\w,\v,\u$ and letters $\li,\lj,\lk$ such that $\w_2=\w\li$, $\v_2=\v\lj$ and $\u_2=\u\lk$. Then,
 \begin{multline*}
  (\w_2\shuffle\v_2)\shuffle\u_2=(\w\li\shuffle\v\lj)\shuffle\u\lk=\big((\w\shuffle\v\lj)\itensor\li+(\w\li\shuffle\v)\itensor\lj\big)\shuffle\u\lk\\
  =\big((\w\shuffle\v\lj)\shuffle\u\lk\big)\itensor\li+\big((\w\li\shuffle\v)\shuffle\u\lk\big)\itensor\lj+\Big(\big((\w\shuffle\v\lj)\li+(\w\li\shuffle\v)\lj\big)\shuffle\u\Big)\itensor\lk\\
  =\big((\w\shuffle\v\lj)\shuffle\u\lk\big)\itensor\li+\big((\w\li\shuffle\v)\shuffle\u\lk\big)\itensor\lj+\big((\w\li\shuffle\v\lj)\shuffle\u\big)\itensor\lk
 \end{multline*}
 Analogously,
 \begin{eqnarray*}
  \w_2\shuffle(\v_2\shuffle\u_2)=\big(\w\shuffle(\v\lj\shuffle\u\lk)\big)\itensor\li+\big(\w\li\shuffle(\v\shuffle\u\lk)\big)\itensor\lj+\big(\w\li\shuffle(\v\lj\shuffle\u)\big)\itensor\lk.
 \end{eqnarray*}
 Thus, since
 \begin{eqnarray*}
  \ell(\w)+\ell(\v\lj)+\ell(\u\lk)=\ell(\w\li)+\ell(\v)+\ell(\u\lk)=\ell(\w\li)+\ell(\v\lj)+\ell(\u)=n,
 \end{eqnarray*}
 we again get $(\w_2\shuffle\v_2)\shuffle\u_2=\w_2\shuffle(\v_2\shuffle\u_2)$ due to the induction hypothesis.
\end{proof}

Going back to the signatures, the dual pairing \eqref{eq:scalar product} in $\mathrm{T}((\mathbb{R}^d))\times \mathrm{T}(\mathbb{R}^d)$ allows us to extract the coefficient of a word in the signature of a path in the following way:
$$\displaystyle{\left\langle \sigma(X),\li_1\li_2\dots\li_n \right\rangle = \int_0^L dX^{i_1}\dots dX^{i_n}}.$$

Both operations, the concatenation and the shuffle products, behave nicely with respect to the signature, as the following two known results describe. The first result, known as the \emph{shuffle identity}, relates the signature of a path with the shuffle product. 
\begin{proposition}[Shuffle identity, \cite{Ree58}]\label{prop:shuffleidentity}
Let $X:[0,L]\longrightarrow \mathbb{R}^d$ be a piecewise continuously differentiable path. Then, for every $\u,\v\in \mathrm{T}(\mathbb{R}^d)$, 
\begin{eqnarray*}
\left\langle \sigma(X), \u\right\rangle \left\langle \sigma(X),\v \right\rangle = \left\langle \sigma(X), \u\shuffle \v\right\rangle. 
\end{eqnarray*}
\end{proposition}
Another important result, known as \emph{Chen's relation}, describes the signature when we concatenate paths. Let us see how the concatenation path is defined. 
\begin{definition}
Let $X,Y:[0,L] \longrightarrow \mathbb{R}^d$ be two piecewise continuously differentiable paths. We define the \emph{concatenation of} $X$ and $Y$ as the path $X\sqcup Y:[0,2L] \longrightarrow \mathbb{R}^d$ given by $X$ on $[0,L]$ and by $Y_{\cdot - L} - Y_0 + X_L$ on $[L,2L]$ (i.e. take $Y$, move it back to $0$ and then move it to the end of $X$). 
\end{definition}
The concatenation product interplays nicely with the concatenation of paths, as the following proposition shows. 
\begin{proposition}[Chen's identity, \cite{C57}]\label{prop:chensidentity}
Let $X,Y:[0,L] \longrightarrow \mathbb{R}^d$ be two piecewise continuously differentiable paths and consider their concatenation $X \sqcup Y:[0,2L] \longrightarrow \mathbb{R}^d$. Then,
$$ \sigma(X \sqcup Y) = \sigma(X)\itensor \sigma(Y).$$
\end{proposition}

We finish this section with an example on how to compute the first terms of the signature of a path.
\begin{example}\label{ex:signature}
Consider the path $X:[0,1]\longrightarrow \mathbb{R}^2$ given by $X^1(t) =t$ and $X^2(t)=t^2$. 
We compute a few terms of its signature. 
\begin{eqnarray*}
\left\langle \sigma(X),\mathtt{1}\right\rangle &=& \int_0^1 dX^1_{r_1} = \int_0^1 1 dt = 1  \hspace{2cm}
\left\langle \sigma(X),\mathtt{2}\right\rangle =  \int_0^1 dX^2_{r_1} = \int_0^1 2t dt = 1 \\
\left\langle \sigma(X),\mathtt{11}\right\rangle &=& \int_0^1\int_0^{r_2} dX^1_{r_1} dX^1_{r_2} = \int_0^1 r_2 dX^1_{r_2}  = \int_0^1 r_2 dr_2 = \frac{1}{2} \\
\left\langle \sigma(X),\mathtt{12}\right\rangle &=& \int_0^1\int_0^{r_2} dX^1_{r_1} dX^2_{r_2} = \int_0^1 r_2 dX^2_{r_2}  = \int_0^1 2r_2^2 dr_2 = \frac{2}{3} \\
\left\langle \sigma(X),\mathtt{21}\right\rangle &=& \int_0^1\int_0^{r_2} dX^2_{r_1} dX^1_{r_2} = \int_0^1 r_2^2 dX^1_{r_2}  = \int_0^1 r_2^2 dr_2 = \frac{1}{3} \\
\left\langle \sigma(X),\mathtt{22}\right\rangle &=& \int_0^1\int_0^{r_2} dX^2_{r_1} dX^2_{r_2} = \int_0^1 r_2^2 dX^2_{r_2}  = \int_0^1 2r_2^3 dr_2 = \frac{2}{4} = \frac{1}{2}\\
\left\langle \sigma(X),\mathtt{222}\right\rangle &=& \int_0^1\int_0^{r_3}\int_0^{r_2} dX^2_{r_1} dX^2_{r_2}dX^2_{r_3}  = \int_0^1\int_0^{r_3} r_2^2 dX^2_{r_2}dX^2_{r_3} = \int_0^1\int_0^{r_3} 2r_2^3 dr_2dX^2_{r_3} = \\ &&\int_0^1 \frac{r_3^4}{2} dX^2_{r_3} = \int_0^1 r_3^5 dr_3 = \left. \frac{r_3^6}{6} \right|_0^1 = \frac{1}{6}  
\end{eqnarray*}
Therefore, the signature of $X$ is of the form
\begin{eqnarray*}
\sigma(X) = \mathtt{1} + \mathtt{2} +\frac{1}{2} \cdot (\mathtt{11}+\mathtt{22}) + \frac{1}{3}\cdot (2\cdot \mathtt{12} + \mathtt{21})+\frac{1}{6}\cdot \mathtt{222} + \dots
\end{eqnarray*}
\end{example}

\section{Signatures under the action of polynomial maps}\label{sec:3}
Let $p:\ \mathbb{R}^d \longrightarrow \mathbb{R}^m$ be a polynomial map given by the polynomials $p_i(x_1,x_2,\dots,x_d)$, for $i=1,2,\dots, m$, with the property that $p(0)=0$. The \emph{degree of the polynomial map} $p$ is the maximum of the degree of the polynomials that define it, $\deg(p)=\max_i \deg(p_i)$. Moreover, we say that a polynomial map is \emph{homogeneous} if the polynomials $p_i$ are homogeneous of the same degree. Finally, we denote by $J_p$ the Jacobian matrix of format $m\times d$ with entries $J_p^{ij} = \partial_j p_i$, for $i\in\{1,2,\dots, m\}$ and $j\in\{1,2,\dots,d\}$. 

Recall the algebra homomorphism $\varphi_d$ defined by:
\begin{eqnarray*}
\begin{array}{rrcl}
\varphi_d: & \mathbb{R}[x_1,x_2,\dots,x_d] & \longrightarrow & \left(\mathrm{T}(\mathbb{R}^d),\shuffle \right) \\
& x_i & \longmapsto & \li \\
& x_{i_1}\cdots x_{i_l} & \longmapsto & \li_1\shuffle \dots \shuffle \li_l
\end{array}
\end{eqnarray*}
As we mention above, by the properties of the shuffle product in $\mathrm{T}(\mathbb{R}^d)$, this map is injective but is not surjective. 

We now consider maps $M_p$, $\varphi$, and $p^*$ such that they complete the diagram \eqref{diagram} that arises in our main question in the following way:
\begin{eqnarray*}
\xymatrix @R=12mm @C=8mm {
 \mathbb{R}[x_1,\dots, x_m]\ar[rr]^{p^*}\ar@{^{(}->}[d]_{\varphi_m} & & \mathbb{R}[x_1,\dots, x_d]\ar@{_{(}->}[d]_{\varphi_d} \\
\left(\mathrm{T}(\mathbb{R}^m),\shuffle\right)\ar@{.>}[rr]^{\exists M_p}  & & \left(\mathrm{T}(\mathbb{R}^d),\shuffle\right) \\
}
\end{eqnarray*}
Notice that the map $M_p$ is unique when restricted to the image of $\varphi_m$,
but not the full $M_p$ on the tensor algebra $\mathrm{T}(\mathbb{R}^m)$.
\begin{definition}\label{def:Mp}
For any polynomial map $p:\ \mathbb{R}^d \to\mathbb{R}^m$ such that $p(0)=0$, let $k_p^{ij} = \varphi_d\left(J^{ij}_p\right)\in \mathrm{T}(\mathbb{R}^d)$, where $J_p^{ij}$ is the $(i,j)$-entry of the Jacobian matrix of $p$. We define the map $M_p: \mathrm{T}(\mathbb{R}^m) \longrightarrow \mathrm{T}(\mathbb{R}^d)$ recursively as follows: 
\begin{eqnarray*}
\begin{array}{l}
M_p(\e) = \e, \text{ for } \e \text{ the empty word, and }\\
M_p(\w \li) = \displaystyle{\sum_{j=1}^d \left( M_p(\w) \shuffle k^{ij}_p\right) \itensor \lj }, \text{ for any word } \w \text{ and any letter } \li\in\{\mathtt{1}\dots,\mathtt{m}\}.
\end{array}
\end{eqnarray*}
\end{definition}

The following result summarizes a few properties of the map $M_p$. We will use these properties to show that the map that $M_p$ as we construct it restricts according to what we need.
\begin{proposition}\label{prop:propertiesMp}
Consider two polynomial maps $p:\mathbb{R}^d \longrightarrow \mathbb{R}^m$ and $q: \mathbb{R}^m \longrightarrow \mathbb{R}^s$, with $p(0)=0$ and $q(0)=0$, and the algebra homomorphisms $\varphi_m$ and $\varphi_d$. Then, we have the following list of properties:
\begin{enumerate}[(I)]
\item \label{itema} $M_p: ( \mathrm{T}(\mathbb{R}^m), \shuffle) \longrightarrow  ( \mathrm{T}(\mathbb{R}^d), \shuffle)$ is an algebra homomorphism. 
\item \label{itemb} For $i=1,\dots, s$, $M_p \left(\varphi_m(q_i)\right) = \varphi_d(q_i\circ p)$, where the $q_i$'s are the polynomials defining the polynomial map $q$.
\item \label{itemc} $k_{q\circ p}^{ij} =\displaystyle{ \sum_{l=1}^m M_p(k_q^{il}) \shuffle k_p^{lj}}$.
\item\label{itemd} $M_{q\circ p} = M_p M_q$.
\item\label{iteme}  If $p$ is a polynomial map of degree $n$, then $M_p\left(\mathrm{T}^k(\mathbb{R}^m)\right) \subseteq \mathrm{T}^{\leq nk}(\mathbb{R}^d)$.
\item\label{itemf}  If $p$ is an homogeneous polynomial map of degree $n$, $M_p\left(\mathrm{T}^k(\mathbb{R}^m)\right) \subseteq \mathrm{T}^{nk}(\mathbb{R}^d)$
\end{enumerate}
\end{proposition}

\begin{proof}
\begin{enumerate}[(I)]
\item We need to show that, for any words $\w_1$ and $\w_2$ in $\mathrm{T}(\mathbb{R}^m)$,
\begin{eqnarray*}
M_p(\w_1 \shuffle \w_2) = M_p(\w_1) \shuffle M_p(\w_2).
\end{eqnarray*}
We proceed by induction on $\ell(\w_1)+\ell(\w_2)$. For $\ell(\w_1)+\ell(\w_2) \leq 1$, at least one of the two words is the empty word $\e$, and so we assume that $\w_2=\e$. Therefore,
\begin{eqnarray*}
M_p(\w_1 \shuffle \w_2 ) = M_p(\w_1 \shuffle \e) = M_p(\w_1) = M_p(\w_1)\shuffle \e  = M_p(\w_1) \shuffle M_p(\w_2).
\end{eqnarray*}
Assume now that the statement is true for any pair of words with sum of lengths at most $n-1$. Let $\u$ and $\v$ be two words such that $\ell(\u)+\ell(\v) = n-1$ and $\la$ and $\lb$ two arbitrary letters.  Then, using Definitions \ref{recdef-shuffle} and \ref{def:Mp}, and the inductive hypothesis (IH),
\begin{multline*}
 \displaystyle{M_p(\u\la \shuffle \v\lb) \stackrel[\ref{recdef-shuffle}]{Def.}{=} M_p\left( (\u\shuffle \v\lb)\itensor \la + (\u\la \shuffle \v)\itensor \lb \right) \stackrel[\ref{def:Mp}]{Def.}{=}  }\\ \displaystyle{\sum_{i=1}^d \left[ M_p(\u \shuffle \v\lb) \shuffle k_p^{ai} + M_p(\u\la \shuffle \v)\shuffle k_p^{bi}\right] \itensor \li} \stackrel[\ref{def:Mp}]{Def.}{=} \\ 
\displaystyle{\sum_{i=1}^d \sum_{j=1}^d \left[ M_p(\u)\shuffle \left((M_p(\v)\shuffle k_p^{bj}) \itensor \li\right) \shuffle k_p^{ai} + \left((M_p(\u)\shuffle k_p^{aj}) \itensor \lj \right)\shuffle M_p(\v)\shuffle k_p^{bi} \right]\itensor \li }\stackrel{(IH)}{=}  \\
 \left(\sum_{i=1}^d (M_p(\u)\shuffle k_p^{ai})\itensor \li \right) \shuffle \left( \sum_{i=1}^d (M_p(\v)\shuffle k_p^{bi})\itensor \li\right) \stackrel[\ref{def:Mp}]{Def.}{=}  M_p(\u\la) \shuffle M_p(\v\lb).
\end{multline*}

\item Assume that $M_p(\li) = \varphi_d(p_i)$, for all $i$. Then, for the monomial $h(x_1,\dots, x_m) =x_1^{n_1}\cdot x_2^{n_2}\cdots x_m^{n_m}$, $\varphi_m(h) = \mathtt{1}^{\shuffle n_1} \shuffle \dots \shuffle \mathtt{m}^{\shuffle n_m}$. Therefore, by the property \eqref{itema},
\begin{multline*}
M_p(\varphi_m(h)) = M_p(\mathtt{1})^{\shuffle n_1} \shuffle \dots \shuffle M_p(\mathtt{m})^{\shuffle n_m} = \varphi_d(p_1)^{\shuffle n_1} \shuffle \dots \shuffle \varphi_d(p_m)^{\shuffle n_m}  =\\  \varphi_d(p_1^{n_1}\cdots p_m^{n_m}) = \varphi_d(h\circ p),
\end{multline*}
and the property \eqref{itemb} follows by linearity. 

We prove now the claim $M_p(\li) = \varphi_d(p_i)$, for all $i$. Since the two maps $p\mapsto M_p(\li)$ and $p\mapsto \varphi_d(p_i)$ are linear, it is enough to prove the claim for the case when $p_i$ is a monomial of the form $p_i=x_1^{n_1}\cdot x_2^{n_2}\cdots x_d^{n_d}$, with at least one of the $n_i$'s non-zero. In this case,
\begin{multline*}
\displaystyle{M_p(\li) = \sum_{j=1}^d \left(M_p(\e)\shuffle k_p^{ij}\right) \itensor \lj = \sum_{j=1}^d k_p^{ij} \itensor \lj \stackrel{*}{=}}\\ 
\displaystyle{\sum_{\substack{j=1 \\ n_j \neq 0}}^d n_j \left( \mathtt{1}^{\shuffle (n_1-\delta_{1j})} \shuffle \dots \shuffle  \mathtt{d}^{\shuffle (n_d-\delta_{dj})}  \right)\itensor \lj } 
 = \mathtt{1}^{\shuffle n_1} \shuffle \dots \shuffle \mathtt{d}^{\shuffle n_d} = \varphi_d(p_i),
\end{multline*}
where $*$ follows by applying enough iterations of the recursive definition of the shuffle product and the fact that $\li^{\shuffle n+1} = (n+1) \li^{\shuffle n} \itensor \li$. 

\item By the \emph{chain rule}, $\displaystyle{J_{q\circ p}^{ij} = \sum_{l=1}^m\left( J_q^{il}\circ p\right) \cdot J_p^{lj}}$. For one term of that sum, by \eqref{itemb}, $\varphi_d(J_q^{il}\circ p) = M_p\left(\varphi_m(J_q^{il})\right) = M_p(k_q^{il})$. Thus,
\begin{eqnarray*}
k_{q\circ p}^{ij} = \varphi_d\left(J_{q\circ p}^{ij}  \right) = \varphi_d\left( \sum_{l=1}^m\left( J_q^{il}\circ p\right) \cdot J_p^{lj}\right) = \sum_{l=1}^m \varphi_d\left( J_q^{il}\circ p\right) \shuffle \varphi_d(J_p^{lj}) =  \sum_{l=1}^m M_p(k_q^{il})\shuffle k_p^{lj}.
\end{eqnarray*}

\item 
We proceed by induction on $\ell(\w)$. 
For a letter $\li$, 
\begin{multline*}
M_p \circ M_q(\li)= M_p(M_q(\li)) = M_p\left(\sum_{j=1}^m k_q^{ij}\itensor \lj \right) = 
\sum_{j=1}^m M_p\left(k_q^{ij}\itensor \lj \right) = \\
\sum_{j=1}^m \left( \sum_{l=1}^d M_p(k_q^{ij}\shuffle k_p^{jl}\right)\itensor \mathtt{l} = 
\sum_{l=1}^d  \left( \sum_{j=1}^m M_p(k_q^{ij}\shuffle k_p^{jl}\right)\itensor \mathtt{l}
\stackrel{\eqref{itemc}}{=} \sum_{l=1}^d k_{q\circ p}^{il}\itensor \mathtt{l} = M_{q\circ p}(\li).
\end{multline*}

Now, we assume that the statement is true for all the words of length at most $n$ and we refer to it as (IH). Let $\w$ be one of these words and $\li$ any letter. Then,
\begin{multline*}
\displaystyle{ M_p \circ M_q(\w\li) = M_p\left( M_q(\w\li)\right) = 
M_p\left( \left(\sum_{j=1}^m M_q(\w)\shuffle k_q^{ij}\right) \itensor \lj \right) = }\\
\sum_{j=1}^m M_p\left((M_q(\w)\shuffle k_q^{ij})\itensor \lj \right)
= \displaystyle{\sum_{j=1}^m \sum_{l=1}^d \left[M_p\left( M_q(\w)\shuffle k_q^{ij}\right)\shuffle k_p^{jl} \right] \itensor \mathtt{l}\stackrel{\eqref{itema}}{=}  } \\
\displaystyle{ \sum_{j=1}^m \sum_{l=1}^d \left[ M_p(M_q(\w))\shuffle M_p(k_q^{ij})\shuffle k_p^{jl}\right] \itensor \mathtt{l}  \stackrel{\text{(IH)}}{=} 
\sum_{j=1}^m \sum_{l=1}^d \left[ M_{q\circ p}(\w))\shuffle M_p(k_q^{ij})\shuffle k_p^{jl}\right] \itensor \mathtt{l}= } \\ 
\sum_{l=1}^d \left[ M_{q\circ p}(\w))\shuffle \left( \sum_{j=1}^m  M_p(k_q^{ij})\shuffle k_p^{jl}\right)\right] \itensor \mathtt{l}   \stackrel{\eqref{itemc}}{=}
\sum_{l=1}^d \left[ M_{q\circ p}(\w) \shuffle k_{q\circ p}^{il}\right] \itensor \mathtt{l} = M_{q\circ p}(\w\li).\\
\end{multline*}

\item We start by noticing that since the polynomial map $p$ has degree $n$, then $\deg(p_i)\leq n$, for all $i$. Thus,  $\deg(J^{ij}_p)\leq n-1$ and $\ell\left(\varphi_d(J_p^{ij})\right) \leq n-1$, for all $i$ and $j$. 

Now, we proceed by induction on $k$. For $k=1$, $\mathrm{T}^1\left(\mathbb{R}^m\right)$ is the set of letters $\{\mathtt{1},\dots, \mathtt{m}\}$. Since for a letter $\li$ in this set 
$\displaystyle{M_p(\li) = \sum_{j=1}^d k^{ij}_p \itensor \lj}$, then $\displaystyle{\ell(M_p(\li)) =\max_j \left\{\ell(k_p^{ij}\itensor \lj) \right\}\leq n}$. Thus, $M_p(\li) \in \mathrm{T}^{\leq n}(\mathbb{R}^d)$. 

Assume that the statement is true for $k$. Any word $\w^\prime\in \mathrm{T}^{k+1}(\mathbb{R}^m)$ can be written as $\w^\prime = \w \itensor \li$, with $\w \in \mathrm{T}^k(\mathbb{R}^m)$ and $\li$ a letter. We analyze the length of $M_p(\w\itensor \li)$. Since $\displaystyle{M_p(\w\itensor \li) = \sum_{j=1}^d \left(M_p(\w)\shuffle k_p^{ij}\right) \itensor \lj}$, it is enough to upper bound the length of the terms appearing in the sum. By the inductive hypothesis, $\ell(M_p(\w))\leq nk$, and since $\ell(k_p^{ij})\leq n-1$, $\ell(M_p(\w)\shuffle k_p^{ij}) \leq nk + n -1$. Therefore, $\ell(M_p(\w\itensor\li))\leq nk+n-1+1 = n(k+1)$.

\item In this case, since $p$ is homogeneous of degree $n$, then $\deg(p_i)=n$, for all $i$. Moreover, $\deg(J_p^{ij})= n-1$, if the variable $x_j$ appears in $p_i$, or zero, otherwise. 

We proceed by induction on $k$. For $k=1$, let $\li$ be a letter in $\{\mathtt{1},\dots, \mathtt{m}\}$. Then, $\displaystyle{M_p(\li) = \sum_{j=1}^d k_p^{ij} \itensor \lj}$. This sum contains only terms $k_p^{ij}\itensor \lj$, which has length exactly $n$, otherwise $k_p^{ij}$ is zero according to our observation about the Jacobian entries above. Therefore, the statement follows. 

Now, assume the statement is true for $k$. Let $\w\in \mathrm{T}^{k}(\mathbb{R}^m)$ be a word and $\li$ a letter. In this case, $\displaystyle{M_p(\w\itensor \li) = \sum_{j=1}^d \left(M_p(\w)\shuffle k_p^{ij}\right)\itensor \lj}$. Again, the terms appearing in this sum have length $nk+n-1+1 = n(k+1)$, which concludes the proof. 
\end{enumerate}
\end{proof}

Once we have these properties, we recall Theorem \ref{thm: mainquestion} and prove it.
\begin{reptheorem}{thm: mainquestion}
There exists an algebra homomorphism $M_p: \left(\mathrm{T}(\mathbb{R}^m),\shuffle\right) \longrightarrow \left(\mathrm{T}(\mathbb{R}^d),\shuffle\right)$ such that its restriction $\left. M_p\right|_{Im(\varphi_m)}$ is the unique algebra homomorphism that makes the following diagram commute:
\begin{eqnarray*}
\xymatrix @R=12mm @C=8mm {
 \mathbb{R}[x_1,\dots, x_m]\ar[rr]^{p^*}\ar@{^{(}->}[d]_{\varphi_m} & & \mathbb{R}[x_1,\dots, x_d]\ar@{_{(}->}[d]_{\varphi_d} \\
Im(\varphi_m)\ar@{.>}[rr]^{\exists !\left. M_p\right|_{Im(\varphi_m)}}  & & Im(\varphi_d) \\
}
\end{eqnarray*}
\end{reptheorem}

\begin{proof}
By \eqref{itema} in Proposition \ref{prop:propertiesMp}, the restriction of $M_p$ to the image $Im(\varphi_m)$ is an algebra homomorphism. Moreover,  due to \eqref{itemb}, we have that $M_p\left(Im(\varphi_m)\right)\subseteq Im(\varphi_d)$. Since we restrict to their images, $\varphi_m$ and $\varphi_d$ are isomorphisms and the map $M_p$ is the unique one making the diagram commute. 
\end{proof}

Let us see now the answer to our main question, which is stated as Theorem \ref{thm: main} 
\begin{reptheorem}{thm: main}
Let $X:[0,L] \longrightarrow \mathbb{R}^d$ be a piecewise continuously differentiable path with $X_0=0$ and let $p:\mathbb{R}^d \longrightarrow \mathbb{R}^m$ be a polynomial map with $p(0)=0$. Then, for all $\w \in \mathrm{T}(\mathbb{R}^m)$, 
\begin{eqnarray*}
\left\langle \sigma (p(X) ), \w \right\rangle = \left\langle \sigma(X), M_p(\w)\right\rangle.
\end{eqnarray*}
Equivalently, $\sigma(p(X)) = M_p^*\left(\sigma(X)\right)$. 
\end{reptheorem}
\begin{proof}
Denote by $Y=p(X)$ and by $\displaystyle{\dot{Y} = \sum_{j=1}^d J_p^{ij}(X)\cdot \dot{X}^j}$. 
Notice that each component $X^j$ of the path equals the first signature component, $\sigma^{1}(X)=X$, and therefore, the entries of the Jacobian matrix can be seen as coefficients of the signature of $X$, 
\begin{eqnarray}\label{equation1}
J_p^{ij}(X_t) = \left\langle \sigma\left(\left.X\right|_{[0,t]}\right), k_p^{ij}\right\rangle.
\end{eqnarray}

We proceed by induction on the length of the word $\w$. 
For a letter $\li$, we have that
\begin{multline*}
\left\langle \sigma(Y),\li\right\rangle  = \int_0^L \left\langle \sigma\left(\left.Y\right|_{[0,t]}\right),\e \right\rangle d\dot{Y}_t^i = \int_0^L d\dot{Y}_t^i  = \sum_{j=1}^d \int_0^L J_p^{ij}(X_t)\dot{X}_t^j dt = \sum_{j=1}^d \int_0^L J_p^{ij}(X_t)dX_t^j \stackrel{\eqref{equation1}}{=} \\  \sum_{j=1}^d \int_0^L \left\langle \sigma\left(\left.X\right|_{[0,t]}\right), k^{ij}_p\right\rangle dX_t^j = \sum_{j=1}^d \left\langle \sigma(X), k_p^{ij}\itensor \lj \right\rangle = \left\langle \sigma(X), M_p(\li)\right\rangle. 
\end{multline*}

Now, assume that the statement is true for all the words of length at most $n$. Let $\w$ be any of these words and $\li$ any letter. By the definition of the signature,
\begin{multline}\label{equation2}
\left\langle \sigma (Y), \w\li  \right\rangle = \int_0^L \left\langle \sigma \left(\left.Y\right|_{[0,t]}\right), \w \right\rangle dY^i_t = \int_0^L \left\langle \sigma \left(\left.Y\right|_{[0,t]}\right), \w \right\rangle \dot{Y^i_t}dt = \\
\sum_{j=1}^d  \int_0^L \left\langle \sigma\left(\left.Y\right|_{[0,t]}\right), \w \right\rangle J_p^{ij}(X_t)\dot{X}^j_tdt = 
 \sum_{j=1}^d  \int_0^L \left\langle \sigma\left(\left.Y\right|_{[0,t]}\right), \w \right\rangle J_p^{ij}(X_t)dX^j_t \stackrel{\eqref{equation1}}{=} \\
  \sum_{j=1}^d  \int_0^L \left\langle \sigma\left(\left.Y\right|_{[0,t]}\right), \w \right\rangle \left\langle \sigma\left(\left.X\right|_{[0,t]}\right), k_p^{ij}\right\rangle dX^j_t.
\end{multline}
Now, apply the inductive hypothesis to $\left\langle \sigma\left(\left.Y\right|_{[0,t]}\right), \w \right\rangle$ in \eqref{equation2}, and then by Chen's identity, Proposition \ref{prop:chensidentity}, 
\begin{multline*}
\left\langle \sigma (Y), \w\li  \right\rangle = \sum_{j=1}^d  \int_0^L \left\langle \sigma\left(\left.X\right|_{[0,t]}\right), M_p(\w) \right\rangle \left\langle \sigma\left(\left.X\right|_{[0,t]}\right), k_p^{ij}\right\rangle dX^j_t = \\
\sum_{j=1}^d  \int_0^L \left\langle \sigma\left(\left.X\right|_{[0,t]}\right), M_p(\w)\shuffle k_p^{ij}\right\rangle dX^j_t = 
\sum_{j=1}^d  \left\langle \sigma(X_t), \left(M_p(\w)\shuffle k_p^{ij}\right) \itensor \lj \right\rangle = \left\langle \sigma(X), M_p(\w\li)\right\rangle.
\end{multline*}
\end{proof}
We finish this section with a generalization of Theorem \ref{thm: main} to polynomial maps that do not satisfy the condition $p(0)=0$ and paths that do not start at the origin.  
\begin{corollary}\label{cor: different origin}
Let $X:[0,L] \longrightarrow \mathbb{R}^d$ be a piecewise continuously differentiable path and let $p:\mathbb{R}^d \longrightarrow \mathbb{R}^m$ be a polynomial map. Consider the map $\tilde{p}$ given by $\tilde{p}(y) = p(y+X_0)-p(X_0)$. Then, for all $\w \in \mathrm{T}(\mathbb{R}^m)$, 
\begin{eqnarray*}
\left\langle \sigma (p(X) ), \w \right\rangle = \left\langle \sigma(X), M_{\tilde{p}}(\w)\right\rangle.
\end{eqnarray*}
\end{corollary}
\begin{proof}
The statement follows using the same argument as in the proof of Theorem \ref{thm: main} if we take into account that in this case, at the end of the step \eqref{equation1}, 
\begin{eqnarray*}
J^{ij}_p(X_t)\dot{X}_t^j dt = J_{\tilde{p}}^{ij}(X_t-X_0)dX_t^j = \left\langle \sigma\left(\left.X\right|_{[0,t]}\right), k_{\tilde{p}}^{ij}\right\rangle.
\end{eqnarray*}
\end{proof}

\subsection{$M_p$ as a half-shuffle homomorphism}\label{Sec:Halfshuffle}
The shuffle product can be seen as the symmetrization of the right half-shuffle, which we define in the following way. Let $\displaystyle{\mathrm{T}^{\geq 1}(\mathbb{R}^d) = \bigoplus_{n\geq 1} \mathrm{T}^n(\mathbb{R}^d)}$ denote the vector space spanned by the non-empty words built from $d$ letters.

\begin{definition}{\cite[Eq. (S2)]{S58},\cite[Def. 1]{FP12},\cite[Sect. 18]{EM53}}\label{def:half-shuffle}
The \emph{right half-shuffle} $\halfshuffle:\mathrm{T}^{\geq 1}(\mathbb{R}^d)\times\mathrm{T}^{\geq 1}(\mathbb{R}^d)\to \mathrm{T}^{\geq 1}(\mathbb{R}^d)$ is recursively given on words as
 \begin{align*}
  \w\halfshuffle\li &:= \w\li,\\
  \w\halfshuffle \v\li &:= (\w\halfshuffle\v+\v\halfshuffle\w)\itensor\li,
 \end{align*}
where $\w,\v$ are words and $\li$ is a letter.
\end{definition}
Therefore, for any non-empty words $\w,\v$
\begin{equation*}
 \w\shuffle \v =\w\halfshuffle\v+\v\halfshuffle\w.
\end{equation*}
Indeed, for non-empty words $\w,\v$ and letters $\li,\lj$, we have
\begin{align*}
 \li\halfshuffle\lj+\lj\halfshuffle\li&=\li\lj+\lj\li,\\
 \w\li\halfshuffle\lj+\lj\halfshuffle\w\li&=\w\li\lj+(\lj\halfshuffle\w+\w\halfshuffle\lj)\itensor\li,\quad\text{and}\\ 
 \w\li \halfshuffle \v\lj + \v\lj \halfshuffle \w\li&= (\w\li\halfshuffle\v+\v\halfshuffle\w\li)\itensor\lj+(\v\lj\halfshuffle\w+\w\halfshuffle\v\lj)\itensor\li,
\end{align*}
in accordance with Definition \ref{recdef-shuffle}.
Thus, the second equation in Definition \ref{def:half-shuffle} can be rewritten as
\begin{equation}\label{eq:halfshuffle_shuffle}
 \w\halfshuffle \v\li= (\w\shuffle\v)\itensor\li.
\end{equation}

It turns out that the right half-shuffle is an example of a more general type of algebras, the Zinbiel algebras. 
\begin{definition}\cite[Eq. (S0)]{S58}, \cite[Eq. (4)]{FP12}
A (right) \emph{Zinbiel algebra} is a vector space $Z$ together with a bilinear map $\varhalfshuffle:\,Z\times Z\to Z$ such that, 	for all $a,b,c\in Z$, 
	\begin{equation*}
		a\varhalfshuffle(b\varhalfshuffle c)=(a\varhalfshuffle b+b\varhalfshuffle a)\varhalfshuffle c.
	\end{equation*}
\end{definition}
We include here the proof of the next result since it is interesting.
\begin{theorem}\cite[Proposition 1.8]{L95}\label{thm:Zinbiel_algebra}
 $(\mathrm{T}^{\geq 1}(\mathbb{R}^d),\halfshuffle)$ is a Zinbiel algebra, i.e. for any non-empty words $\w$, $\v$, and $\u$,
 \begin{eqnarray}\label{eq:halfshuffle_Zinbiel}
\w\halfshuffle (\v\halfshuffle\u)= (\w\shuffle\v)\halfshuffle\u.
 \end{eqnarray}
\end{theorem}

\begin{proof}
 Let $\w$, $\v$, and $\u$ be non-empty words and $\li$ be an arbitrary letter. By the definition of the half-shuffle and Equation \eqref{eq:halfshuffle_shuffle}, we have
          \begin{equation*}
           \w\halfshuffle(\v\halfshuffle\li)=\w\halfshuffle\v\li=(\w\shuffle\v)\itensor\li=(\w\shuffle\v)\halfshuffle\li.
          \end{equation*}
          Using Equation \eqref{eq:halfshuffle_shuffle} and associativity of the shuffle product, Proposition \ref{prop:shuffle_associativity}, we obtain that
          \begin{eqnarray*}
            \w\halfshuffle (\v\halfshuffle\u\li)&=\w\halfshuffle\big((\v\shuffle\u)\itensor\li\big)=\big(\w\shuffle(\v\shuffle\u)\big)\itensor\li=\big((\w\shuffle\v)\shuffle\u\big)\itensor\li=(\w\shuffle\v)\halfshuffle\u\li.
          \end{eqnarray*}
\end{proof}
In fact, it is known that $(\mathrm{T}^{\geq 1}(\mathbb{R}^d),\halfshuffle)$ is free in this case.
\begin{theorem}\label{thm:free_Zinbiel}\cite[page 19]{S58}\cite[Proposition 1.8]{L95}
 Indeed, $(\mathrm{T}^{\geq 1}(\mathbb{R}^d),\halfshuffle)$ is the free Zinbiel algebra over $\mathbb{R}^d$.
\end{theorem}

This means that for any Zinbiel algebra $(Z,\varhalfshuffle)$ and any linear map $\linearmap:\,\mathbb{R}^d\to Z$, there is a unique homomorphism $\hsm_\linearmap:\,(\mathrm{T}^{\geq 1}(\mathbb{R}^d),\halfshuffle)\to(Z,\varhalfshuffle)$ such that $\linearmap=\hsm_\linearmap\circ\iota$, where $\iota:\,\mathbb{R}^d\to\mathrm{T}^{\geq 1}(\mathbb{R}^d)$ is the canonical embedding. This is known as the \emph{universal property of the free Zinbiel algebra} and is described in Diagram \ref{Diagram}. We call $\hsm_\linearmap$ the \emph{unique extension of $B$ to a Zinbiel homomorphism}.
\begin{diag}
	\centering
	\begin{tikzcd}[column sep=large]
		\mathbb{R}^d \arrow[r, "\iota"] \arrow[dr, "\linearmap"] & (\mathrm{T}^{\geq 1}(\mathbb{R}^d),\halfshuffle) \arrow[d, "\hsm_\linearmap"]\\
		& (Z,\varhalfshuffle)
	\end{tikzcd}
	\caption{Universal property of the free Zinbiel algebra}\label{Diagram}
\end{diag}

\begin{proof}
	Define the linear map $\hsm_\linearmap:\,\mathrm{T}^{\geq 1}(\mathbb{R}^d)\to Z$ recursively by
	 \begin{equation*}
	  \hsm_\linearmap \li:=L \li,\qquad\hsm_\linearmap \v\li:=\hsm_\linearmap \v\varhalfshuffle\hsm_\linearmap \li,
	 \end{equation*}
where we identified $\mathbb{R}^d$ with the letters in $\mathrm{T}^{\geq 1}(\mathbb{R}^d)$. Since $\v\li=\v\halfshuffle \li$, this is the only candidate for a map with the desired properties. It remains to show that it is indeed a homomorphism of Zinbiel algebras.
	
	 By definition, it holds that $\hsm_\linearmap \li\lj=\hsm_\linearmap \li \varhalfshuffle \hsm_\linearmap \lj$. Thus, assume that $\hsm_\linearmap \mathtt{x}\mathtt{y}=\hsm_\linearmap \mathtt{x}\varhalfshuffle \hsm_\linearmap \mathtt{y}$ holds for all nonempty words $\mathtt{x}$, $\mathtt{y}$ such that $|\mathtt{x}|+|\mathtt{y}|=n$. Then, for all nonempty words $\w$ and $\v$ such that $|\w|+|\v|=n$, we have
	 \begin{multline*}
	 	\hsm_\linearmap \w\varhalfshuffle\hsm_\linearmap \v\li
	 	=\hsm_\linearmap \w\varhalfshuffle(\hsm_\linearmap \v\varhalfshuffle\hsm_\linearmap \li)=(\hsm_\linearmap \w \varhalfshuffle \hsm_\linearmap \v+\hsm_\linearmap \v\varhalfshuffle \hsm_\linearmap \w)\varhalfshuffle \hsm_\linearmap \li=\\
	 	\hsm_\linearmap(\w\halfshuffle \v+\v\halfshuffle \w)\varhalfshuffle\hsm_\linearmap \li=\hsm_\linearmap((\w\halfshuffle \v+\v\halfshuffle \w)\halfshuffle \li)
	 	=\hsm_\linearmap(\w\halfshuffle(\v\halfshuffle \li))=\hsm_\linearmap(\w\halfshuffle \v\li),
	 \end{multline*}
        and $\hsm_\linearmap \u\li:=\hsm_\linearmap\u\varhalfshuffle\hsm_\linearmap \li$, for all nonempty words $\u$ with $|\u|=n$ by definition. The claim follows by induction over $n$.
\end{proof}
The following result describes the relation between the map $M_p$ and the half-shuffle. 
\begin{theorem}
The restriction of $M_p$ to $\mathrm{T}^{\geq 1}(\mathbb{R}^d)$, $\left.M_p\right|_{{\mathrm{T}^{\geq 1}(\mathbb{R}^d)}}$, is the unique half-shuffle homomorphism such that $M_p(\li)=\varphi_d(p_i)$.
\end{theorem}

\begin{proof}
 Using \eqref{eq:halfshuffle_shuffle} and the definition of $M_p$, we get
 \begin{equation*}
  M_p(\w\li)=\sum_{j=1}^d \left( M_p(\w) \shuffle k^{ij}_p\right) \itensor \lj=M_p(\w)\halfshuffle\Big(\sum_{j=1}^d k^{ij}_p \itensor \lj\Big)=M_p(\w)\halfshuffle M_p(\li),
 \end{equation*}
 and thus the statement follows immediately from the proof of Theorem \ref{thm:free_Zinbiel}.
\end{proof}

We finish this section with a generalization of Theorem \ref{thm: main} in terms of Zinbiel algebras stated in the following result.
\begin{theorem}\label{theorem:signaturesandhsm}
 Let $X:[0,L] \longrightarrow \mathbb{R}^d$ be a piecewise continuously differentiable path and $\linearmap:\,\mathbb{R}^m\to(\mathrm{T}^{\geq 1}(\mathbb{R}^d),\halfshuffle)$ be a linear map. Then, the signature of the path 
 \begin{equation*}
 Y:[0,L]\longrightarrow\mathbb{R}^m, \quad Y_t^i:=\left\langle\sigma\left(X|_{[0,t]}\right),\linearmap\li\right\rangle,
 \end{equation*}
 is a linear transformation of the signature of $X$, namely 
 \begin{equation*}
  \left\langle \sigma (Y), \w \right\rangle = \left\langle \sigma(X), \hsm_B\w\right\rangle,
 \end{equation*}
 where $\hsm_B$ is the unique extension of $B$ to a Zinbiel homomorphsim.
\end{theorem}

Before proving this result, we introduce some notation. We denote by $X_{0s}^{z}$ the coefficient of $z$ in the signature of the path $X$ restricted to the interval $[0,s]$. That is, 
 \begin{equation*}
  X_{0s}^{z}:=\langle\sigma(X{\restriction_{[0,s]}}),z\rangle.
 \end{equation*}
Note that $X_{0s}^{\li}=X_s^i-X_0^i$. Then, the path $Y$ introduced in Theorem \ref{theorem:signaturesandhsm} is given by $Y_s^{\li}=X_{0s}^{B\li}$. Moreover, for any letter $\li$, we define the maps $\Tminus{\li},\Tplus{\li}:\,\mathrm{T}(\mathbb{R}^d)\to\mathrm{T}(\mathbb{R}^d)$ to be the unique linear maps given recursively by $\Tplus{\li}\w=\w\li$ and by $\Tminus{\li}\w\lj=\delta_{\li\lj}\w$ with $\Tminus{\li}\e=0$, respectively, for any word $\w$.  

These two maps allows us to define the right half-shuffle as shown in the following technical result. 
\begin{lemma}\label{lemma:TminusTplus}
 For any $x,y\in\mathrm{T}^{\geq 1}(\mathbb{R}^d)$, we have $x\halfshuffle y=\displaystyle{\sum_{\li=\mathtt{1}}^{\mathtt{d}}\Tplus{\li}(x\shuffle\Tminus{\li}y)}$.
\end{lemma}
\begin{proof}
 This is just a reformulation of \eqref{eq:halfshuffle_shuffle} in the following way. For any word $\v$, any non-empty word $\w$, and any letter $\lj$, we have that 
 \begin{equation*}
  \sum_{\li=\mathtt{1}}^{\mathtt{d}}\Tplus{\li}(\w\shuffle\Tminus{\li}\v\lj)=\Tplus{\lj}(\w\shuffle\v)=(\w\shuffle\v)\itensor\lj=\w\halfshuffle\v\lj,
 \end{equation*}
 where the last equality follows from \eqref{eq:halfshuffle_shuffle}. Then, the general statement for any $x,y\in\mathrm{T}^{\geq 1}(\mathbb{R}^d)$ follows from (bi)linearity.
\end{proof}
Now we are ready to prove Theorem \ref{theorem:signaturesandhsm}.

\begin{proof}[Proof of Theorem \ref{theorem:signaturesandhsm}]
For better readability, we put $\hsm:=\hsm_B$. First note that by the definition of the signature and the fact that $X$ is continuously differentiable almost everywhere, we have
\begin{equation*}
 Y_s^{\li}=X_{0s}^{B\li}=X_{0s}^{\hsm\li}=\sum_{i=1}^d\int_0^s X_{0t}^{\Tminus{\li}\hsm\li}\mathrm{d}X_t^i=\sum_{i=1}^d\int_0^s X_{0t}^{\Tminus{\li}\hsm\li}\dot{X}_t^i\mathrm{d}t
\end{equation*}
and thus, for almost all $s\in[0,T]$, 
\begin{equation*}
 \dot{Y}_s^{\li}=\sum_{i=1}^d X_{0t}^{\Tminus{\li}\hsm\li}\dot{X}_t^i.
\end{equation*}
Following an inductive argument, assume now that $X_{0s}^{\hsm\w}=Y_{0s}^{\w}$ holds for some word $\w$. Then,
 \begin{align*}
  Y_{0s}^{\w\li}=\int_0^s Y_{0t}^{\w}\mathrm{d}Y^i_t=\int_0^s Y_{0t}^{\w}\dot{Y}_t^i\mathrm{d}t=\sum_{l=1}^d\int_0^s X_{0t}^{\hsm\w}X_{0t}^{\Tminus{l}\hsm\li}\dot{X}_t^l\mathrm{d}t\\
  =\int_0^s X_{0t}^{\sum_{\mathtt{l}=\mathtt{1}}^{\mathtt{d}}\hsm\w\shuffle\Tminus{\mathtt{l}}\hsm\li}\mathrm{d}X_{0t}^{\mathtt{l}}=X_{0t}^{\sum_{\mathtt{l}=\mathtt{1}}^{\mathtt{d}}\Tplus{\mathtt{l}}(\hsm\w\shuffle\Tminus{\mathtt{l}}\hsm\li)}=X_{0t}^{\hsm\w\halfshuffle\hsm\li}=X_{0t}^{\hsm\w\li},
 \end{align*}
were we used Lemma \ref{lemma:TminusTplus} and the fact that $\hsm$ is a homomorphism of Zinbiel algebras.
\end{proof}

Theorem \ref{theorem:signaturesandhsm} can also be directly shown using
\begin{equation}\label{eq:halfshuffle_identity}
  \int_0^s X_{0t}^{x}\mathrm{d} X_{0t}^{y}=X_{0s}^{x\halfshuffle y}
\end{equation}
for any $x,y\in\mathrm{T}^{\geq 1}(\mathbb{R}^d)$, a relation which is quite fundamental for an algebraic understanding of the signature and was mentioned already for example in \cite{GK} right after equation~(6). Conversely, starting from Theorem \ref{theorem:signaturesandhsm}, equation \eqref{eq:halfshuffle_identity} is immediate with the choice $B\mathtt{1}=x, B\mathtt{2}=y$.

%
The following example illustrates the results presented in this section.
\begin{example}
For a given path $X:[0,L]\to\mathbb{R}^3$, let 
\begin{equation*}
 Y=\left(\operatorname{Area}(X^2,X^3),\operatorname{Area}(X^3,X^1),\operatorname{Area}(X^1,X^2)\right)
\end{equation*}
denote the area path of X, where $\operatorname{Area}(X^i,X^j)_t=\int_0^t\int_0^s\mathrm{d}X^i_u\mathrm{d}X^j_s-\int_0^t\int_0^s\mathrm{d}X^j_u\mathrm{d}X^i_s$. Let us compute $\langle\sigma(Y),\mathtt{12}-\mathtt{21}\rangle$ in the special case that $X:[0,1]\to\mathbb{R}^3,\,X(t)=(t,t^2,t^3)$. We have $Y_s^i=X_{0s}^{B\li}$ where
\begin{equation*}
 B\mathtt{1}=\mathtt{23}-\mathtt{32},\quad B\mathtt{2}=\mathtt{31}-\mathtt{13},B\mathtt{3}=\mathtt{12}-\mathtt{21}
\end{equation*}
We need to compute $\hsm_B(\mathtt{12}-\mathtt{21})=B\mathtt{1}\halfshuffle B\mathtt{2}-B\mathtt{2}\halfshuffle B\mathtt{1}$. To this end,
\begin{align*}
(\mathtt{23}-\mathtt{32})\halfshuffle(\mathtt{31}-\mathtt{13})&=2\cdot\mathtt{2331}-2\cdot\mathtt{3321}-\mathtt{2313}-\mathtt{2133}-\mathtt{1233}+\mathtt{3213}+\mathtt{3123}+\mathtt{1323},\\
(\mathtt{31}-\mathtt{13})\halfshuffle(\mathtt{23}-\mathtt{32})&=\mathtt{3123}+\mathtt{3213}+\mathtt{2313}-\mathtt{1323}-\mathtt{1233}-\mathtt{2133}-\mathtt{3132}-2\cdot\mathtt{3312}\\
&\quad+2\cdot\mathtt{1332}+\mathtt{3132}.
\end{align*}
Thus,
$\hsm_B(\mathtt{12}-\mathtt{21})=2\cdot\left(-\mathtt{1323}+\mathtt{1332}+\mathtt{2313}-\mathtt{2331}+\mathtt{3312}-\mathtt{3321}\right)$. Since in our special case of $X(t)=(t,t^2,t^3)$ it holds that \cite[Example 2.2]{AFS18}
\begin{equation*}
 \int_0^1\int_0^u\int_0^t\int_0^s\mathrm{d}X^i_r\mathrm{d}X^j_s\mathrm{d}X^k_t\mathrm{d}X^l_u=\frac{j\cdot k\cdot l}{(j+i)(k+j+i)(l+k+j+i)},
\end{equation*}
we get that
\begin{align*}
 \langle&\sigma(Y),\mathtt{12}-\mathtt{21}\rangle\\
 &=2\cdot\left(-\frac{3\cdot 2\cdot3}{4\cdot6\cdot9}+\frac{3\cdot 3\cdot 2}{4\cdot 7\cdot 9}+\frac{3\cdot 1\cdot 3}{5\cdot 6 \cdot 9}-\frac{3\cdot3\cdot 1}{5\cdot8\cdot9}+\frac{3\cdot1\cdot 2}{6\cdot7\cdot9}-\frac{3\cdot2\cdot 1}{6\cdot8\cdot9}\right)\\
 &=-\frac{1}{315}\approx-0.00317
\end{align*}
as the desired result.
\end{example}

\section{Examples and consequences}\label{Sec:4}
Let us start with an \emph{easy example} to illustrate how Theorem \ref{thm: main} works, and also the property \eqref{iteme} in Proposition \ref{prop:propertiesMp}.
\begin{example}\label{Example1}
Consider the polynomial map $p:\mathbb{R}^2 \longrightarrow \mathbb{R}^3$ given by the polynomials $p_1 = x^2$, $p_2=y^3$ and $p_3=x-y$. Consider also the path in Example \ref{ex:signature}, $X:[0,1] \longrightarrow \mathbb{R}^2$ given by $X^1(t)=t$ and $X^2(t)=t^2$. We want to compute a few terms in the signature of the path $p(X)$. 

We start computing the Jacobian matrix and its image under $\varphi$:
\begin{eqnarray*}
\left(J_p^{ij}\right)_{i,j} = \left( 
\begin{array}{cc}
2x & 0 \\
0 & 3y^2 \\
1 & -1
\end{array} \right) \hspace{1cm} \text{ and } \hspace{1cm} 
\left(k_p^{ij}\right)_{i,j} = \left(\varphi\left(J_p^{ij}\right)\right)_{i,j}= \left( 
\begin{array}{cc}
2\cdot \mathtt{1} & 0 \\
0 & 6\cdot \mathtt{22} \\
\e & - \e
\end{array} \right).
\end{eqnarray*}
Notice that $\varphi(3y^2) = 3\cdot \mathtt{2}\shuffle \mathtt{2} = 6 \cdot \mathtt{22}$.
We use Definition \ref{def:Mp} to compute the image of a few words:
\begin{eqnarray*}
\begin{array}{l}
M_p(\mathtt{1}) = 2\cdot \mathtt{11}, \hspace{1cm} M_p(\mathtt{2}) = 6\cdot \mathtt{222}, \hspace{1cm} M_p(\mathtt{3}) = \mathtt{1} - \mathtt{2},\\[2ex]
M_p(\mathtt{33}) = (\mathtt{1}-\mathtt{2})\itensor \mathtt{1} - (\mathtt{1}-\mathtt{2})\itensor \mathtt{2} = \mathtt{11} + \mathtt{22} - \mathtt{12} - \mathtt{21}.
\end{array}
\end{eqnarray*}
We observe that for any word $\w$ above, $\ell\left(M_p(\w)\right) \leq 3\cdot \ell(\w)$. This is due to property $\eqref{iteme}$ in Proposition \ref{prop:propertiesMp} since $\deg(p)=3$.
Now, applying Theorem \ref{thm: main} and looking at the signature terms computed in Example \ref{ex:signature}, we obtain that
\begin{eqnarray*}
\begin{array}{l}
\displaystyle{\left\langle \sigma\left(p(X)\right),\mathtt{1}\right\rangle =  \left\langle \sigma(X),M_p(\mathtt{1})\right\rangle =\left\langle \sigma(X), 2\cdot \mathtt{11}\right\rangle = 2\cdot \frac{1}{2} =1, }\\[2ex]
\displaystyle{\left\langle \sigma\left(p(X)\right),\mathtt{2}\right\rangle =  \left\langle \sigma(X),M_p(\mathtt{2})\right\rangle =\left\langle \sigma(X), 6\cdot \mathtt{222}\right\rangle = 6\cdot \frac{1}{6} =1, }\\[2ex]
\displaystyle{\left\langle \sigma\left(p(X)\right),\mathtt{3}\right\rangle =  \left\langle \sigma(X),M_p(\mathtt{3})\right\rangle =\left\langle \sigma(X), \mathtt{1}-\mathtt{2}\right\rangle = 1-1=0, } \text{ and }\\[2ex]
\displaystyle{\left\langle \sigma\left(p(X)\right),\mathtt{33}\right\rangle =  \left\langle \sigma(X),M_p(\mathtt{33})\right\rangle =\left\langle \sigma(X), \mathtt{11}+\mathtt{22}-\mathtt{12}-\mathtt{21}\right\rangle = 0. }\\[2ex]
\end{array}
\end{eqnarray*}
\end{example}

This second example is more generic and shows the property \eqref{itemf} in Proposition \ref{prop:propertiesMp}.
\begin{example}\label{Example2}
Let $X$ be any piecewise continuously differentiable path in $\mathbb{R}^2$. Consider the polynomial map $p:\mathbb{R}^2\longrightarrow \mathbb{R}^3$ given by $p(x,y) = (x^2,xy,y^2)$, and fix $k=2$. 

By Theorem \ref{thm: main}, the coefficient of $\w$ in $\sigma\left(p(X)\right)$ is given by the coefficient of $M_p(\w)$ in $\sigma(X)$. Moreover, by Proposition \ref{prop:propertiesMp} \eqref{itemf}, the words appearing in $M_p(\w)$ have length exactly 4. One way of storing $\sigma^{(2)}\left(p(X)\right)$ is using a matrix that encodes the coefficients in $M_p(\w)$, 
\begin{eqnarray*}
\begin{array}{r}
\sigma^{(2)}\left(p(X)\right) \\[2ex]
 \left[ \begin{array}{c} \mathtt{11} \\ \mathtt{12} \\ \mathtt{13} \\ \mathtt{21} \\ \mathtt{22} \\ \mathtt{23} \\ \mathtt{31} \\ \mathtt{32} \\ \mathtt{33}  \end{array} \right] \end{array} 
 \begin{array}{c}
\hspace{0.5cm} \\[2ex] \leftrightsquigarrow  \left[\begin{array}{cccccccccccccccc} 
 12 & 0 & 0 & 0 & 0 & 0 & 0 & 0 & 0 & 0 & 0 & 0 & 0 & 0 & 0 & 0 \\
 6 & 0 & 2 & 0 & 2 & 0 & 0 & 0 & 2 & 0 & 0 & 0 & 0 & 0 & 0 & 0 \\
0 & 0 & 0 & 4 & 0 & 4 & 0 & 0 & 0 & 4 & 0 & 0 & 0 & 0 & 0 & 0  \\
 0 & 0 & 4 & 0 & 4 & 0 & 0 & 0 & 4 & 0 & 0 & 0 & 0 & 0 & 0 & 0\\
 0 & 0 & 0 & 2 & 0 & 2 & 2 & 0 & 0 & 2 & 2 & 0 & 2 & 0 & 0 & 0\\ 
 0 & 0 & 0 & 0 & 0 & 0 & 0 & 4 & 0 & 0 & 0 & 4 & 0 & 4 & 0 & 0\\ 
0 & 0 & 0 & 0 & 0 & 0 & 4 & 0 & 0 & 0 & 4 & 0 & 4 & 0 & 0 & 0\\
 0 & 0 & 0 & 0 & 0 & 0 & 0 & 2 & 0 & 0 & 0 & 2 & 0 & 2 & 6 & 0\\
 0 & 0 & 0 & 0 & 0 & 0 & 0 & 0 & 0 & 0 & 0 & 0 & 0 & 0 & 0 & 12 \\
 \end{array}
  \right] \end{array}
  \begin{array}{l}
\sigma^{(4)}(X) \\[2ex] \left[  \begin{array}{c} \mathtt{1111} \\ \mathtt{1112} \\ \mathtt{1121}  \\ \mathtt{1122} \\ \mathtt{1211} \\ \mathtt{1212} \\ \mathtt{1221} \\ \mathtt{1222} \\ \mathtt{2111} \\ \mathtt{2112} \\ \mathtt{2121} \\ \mathtt{2122} \\ \mathtt{2211} \\ \mathtt{2212} \\  \mathtt{2221} \\ \mathtt{2222}  \end{array} \right].\end{array}
\end{eqnarray*}

More generally, by Proposition \ref{prop:propertiesMp} \eqref{itemf}, for any word $\w$ in the alphabet $\{\mathtt{1},\mathtt{2},\mathtt{3}\}$ with $\ell(\w)=k$, $M_p(\w)$ is a sum of words in the alphabet $\{\mathtt{1},\mathtt{2}\}$ of length $2k$. Applying Theorem \ref{thm: main}, the information of $\sigma^{(k)}\left(p(X)\right)$ can be stored in terms of $\sigma^{(2k)}(X)$. In fact, there exists a matrix $M$ of format $3^k \times 2^{2k}$ that describes the \emph{change of coordinates} in the following sense. Fix an order on the words, for instance, \emph{lexicographic order} as above. The coefficients of $\sigma^{(k)}(p(X))$ and the coefficients of $\sigma^{(2k)}(X)$ are related as 
\begin{eqnarray*}
\begin{array}{r}
\sigma^{(k)}\left(p(X)\right) \\[2ex]
 \left[ \begin{array}{c} \overbrace{\mathtt{11}\dots \mathtt{1}}^{k} \\ \mathtt{1}\dots \mathtt{12} \\ \vdots \\ \mathtt{3}\dots \mathtt{32} \\ \underbrace{\mathtt{3}\dots \mathtt{33}}_{k} \end{array} \right] \end{array} \begin{array}{c} \hspace{0.1cm} \\[2ex] = M \end{array}  \begin{array}{l}
 \sigma^{(2k)}(X) \\[2ex]
 \left[  \begin{array}{c} \overbrace{\mathtt{11}\dots \mathtt{1}}^{2k} \\ \mathtt{1}\dots \mathtt{12} \\ \vdots \\ \mathtt{2}\dots \mathtt{21} \\ \underbrace{\mathtt{2}\dots \mathtt{22}}_{2k} \end{array} \right],\end{array}
\end{eqnarray*}
where the row indexed by $\w$ in $M$ is given by the coefficients of $M_p(\w)$. 
Moreover, we want to point out the following properties of this homomorphism $M_p$:
\begin{itemize}
\item For $\w = \underbrace{\mathtt{1}\dots \mathtt{1}}_{k\ times}$, 
$\displaystyle{M_p(\w) = \frac{(2k)!}{k!} \cdot \w \itensor \w}$. 
\item For $\w = \underbrace{\mathtt{2}\dots \mathtt{2}}_{k\ times}$, 
$\displaystyle{M_p(\w) = k! \cdot  \underbrace{\mathtt{1}\dots \mathtt{1}}_{k\ times} \shuffle \underbrace{\mathtt{2}\dots \mathtt{2}}_{k\ times}}$.
\item For $\w = \underbrace{\mathtt{3}\dots \mathtt{3}}_{k\ times}$, 
$\displaystyle{M_p(\w) = \frac{(2k)!}{k!}\underbrace{\mathtt{2}\dots \mathtt{2}}_{2k\ times}}$. 
\end{itemize}
\end{example}
The rest of this section is dedicated to analize the consequences of Theorems \ref{thm: mainquestion} and \ref{thm: main} in several particular cases. 
We start looking at the case in which $X$ is itself a polynomial map and we need the following definition.
\begin{definition}
For an element $\la\in \mathrm{T}((\mathbb{R}^d))$, with zero coefficient for the empty word $\e$, we define the \emph{concatenation product exponential} of $\la$ as 
\begin{eqnarray*}
\exp_\itensor (\la) := \sum_{n\geq 0} \frac{\la^{\itensor n}}{n!}.
\end{eqnarray*}
\end{definition}
More information on this exponential map and its inverse, the \emph{logarithm}, can be found in \cite{Reu93}.

\begin{corollary}\label{cor:polynomialpath}
Let $X:[0,L] \longrightarrow \mathbb{R}^d$ be a polynomial map, with $L\in \mathbb{R}$, $L\geq 1$. Then, for any $\w \in \mathrm{T}(\mathbb{R}^d)$,
\begin{eqnarray*}
\left\langle \sigma(X), \w\right\rangle =\left\langle \exp_\itensor (L\cdot \mathtt{1}), M_{\tilde{X}}(\w)\right\rangle,
\end{eqnarray*}
where $\tilde{X}(y) = X(y)-X_0$. Equivalently, $\sigma(X) = M^*_{\tilde X}(\exp_\itensor(L\cdot \mathtt{1}))$.
\end{corollary}
\begin{proof}
Let $Y:[0,L]\longrightarrow \mathbb{R}$ be the path given by $Y(t) = t$. Then, $\sigma(Y) = \exp_\itensor (L\cdot \mathtt{1})$. The statement follows by Corollary \ref{cor: different origin} applied to the path $Y$ and the polynomial map $X$. 
\end{proof}

The next result looks at the case when $M_p^*(\sigma(X))=0$ from the perspective of the polynomial map and of the piecewise continuously differentiable path. We introduce first two concepts. 
Given a polynomial map $p$, we define the \emph{ideal generated by $p$} as the ideal $I_p$ generated by the polynomials that define the map $p$, i.e.  $I_p = \left\langle p_1,\dots, p_m\right\rangle$. 
Moreover, we define a \emph{tree-like path} as a path $X$ such that $\sigma(X)  = 0$. This definition is the characterization given by B.M. Hambly and T. J. Lyons, \cite{HL10}, and a more general topological definition can be found in \cite[Definition~1.1]{BGLY16}.
\begin{remark}
A very simple example of a tree-like path is a concatenation of paths $A\sqcup B\sqcup C \sqcup D$ such that the paths A and B (resp.\ C and D) are of the same shape, but parametrized in the opposite direction. When we compute the integrals on such a path, we get cancellations and the signature of the path does not see the $A\sqcup B\sqcup C\sqcup D$ loop, i.e.\ $\sigma(A\sqcup B\sqcup C\sqcup D)=\mathbf{e}$.
\end{remark}
The following result describes the situation in which the (image of the) path lies in the zeros of the ideal $I_p$, for some polynomial map $p$.
\begin{corollary}\label{cor:imagepath}
Let $p: \mathbb{R}^d \longrightarrow \mathbb{R}^m$ be a polynomial map with associated ideal $I_p$. We define the polynomial map $\tilde{p}(y)= p(y+X_0) - p(X_0)$, for which $\tilde{p}(0)=0$. 
\begin{itemize}
\item If $X:[0,L]\longrightarrow \mathbb{R}^d$ is a piecewise continuously differentiable path such that $X(t)\in \mathcal{V}(I_p)$ for all $t\in [0,L]$, then $M^*_{\tilde{p}}(\sigma(X)) = \mathbf{e}$.
\item Conversely, if $M^*_{\tilde{p}}(\sigma(X)) = \mathbf{e}$, for some piecewise continuously differentiable path $X:[0,L]\longrightarrow \mathbb{R}^d$, then $p(X)$ is tree-like.
\end{itemize} 
\end{corollary}

The last consequence is that the dual map $M_p^*$ behaves nicely with respect to the concatenation of signatures. 
\begin{corollary}\label{cor:dualmap}
Let $p:\mathbb{R}^d \longrightarrow \mathbb{R}^m$ be a polynomial map with $p(0)=0$ and let $X,Y:[0,L]\longrightarrow \mathbb{R}^d$ be two piecewise continuously differentiable paths with $X_0=Y_0=0$. Then, 
$$
M_p^*(\sigma(X) \itensor\sigma(Y)) = M_p^*(\sigma(X)) \itensor M_q^*(\sigma(Y)),$$
where $ q(y)=p(y+X_L-X_0)-p(X_L-X_0)$.
\end{corollary}

\section{Applications and future work}\label{sec:5}

The results presented in this paper solve an algebraic question that arises from the signatures of paths, an object commonly studied in stochastic analysis. There are several interesting problems that do not fit on the algebraic flavour of this paper. We summarize them in the following list.
\begin{enumerate}
\item[(A)] In comparison with the results presented in \cite{PSS18}, we would like to explore a non-linear version of their approach using dictionaries. The idea is that if we have a family of generic paths, $\chi$, for which we know the signature and a polynomial map $p$, then Theorem \ref{thm: main} allows us to compute the signature of all the paths in $p(\chi)$. 

For instance, Example \ref{Example2} shows that we can compute the signature of $p(X)$ for any piecewise continuously differentiable path $X$ by multiplying the signature of $X$ by a matrix at each level. Therefore, we have the following question:
\begin{center}
\emph{Is it possible to understand $\sigma(p(X))$ in terms of $\sigma(X)$ and $M_p$ \\ in the language of tensors?}
\end{center}


\item[(B)] Another line of future research is focused on the map $M_p$. Since it is defined from a polynomial map without involving any piecewise continuously differentiable path and gives us the commuting diagram \eqref{diagram}, we intuit that it is worth to look for more interesting properties. For that, we should look to the \emph{big picture} involving the Hopf algebra structures, as well as other constructions. 

In this direction, at the end of Example \ref{Example2} we describe combinatorially $M_p(\w)$ for some particular words $\w$. A more general question would be the following:
 \begin{quote}
For which words $\w$ and polynomial maps $p$ is there a \emph{non-recursive} combinatorial formula for $M_p(\w)$?
 \end{quote}
Answering this question could be very useful from the computational perspective. 
\end{enumerate}

\section*{Acknowledgements}
The authors would like to thank Bernd Sturmfels, Francesco Galuppi, Joscha Diehl, Mateusz Micha\l{}ek, and Max Pfeffer for their fruitful conversations and observations. 
The collaboration between the co-authors would not have been possible without the financial support from the research institute MPI-MiS Leipzig (Germany). R.P. is currently supported by European Research Council through CoG-683164.

 \bibliographystyle{alpha}
 \bibliography{biblio}
\end{document}